\documentclass[12pt,leqno]{article}
\usepackage{lineno}
\usepackage{amsmath,amsthm}
\usepackage{txfonts}
\usepackage{bbm}
\usepackage{amssymb}
\usepackage{amsbsy}
\usepackage{mathrsfs}
\usepackage{graphicx}
 \usepackage{color,xcolor}

\allowdisplaybreaks
\theoremstyle{plain}
\theoremstyle{plain}
\theoremstyle{remark}  \newtheorem{remark}{\noindent\mbox{Remark}}
\theoremstyle{remark}  
\theoremstyle{definition}
\theoremstyle{plain}\newtheorem{lemma}{\noindent\mbox{Lemma}}
\theoremstyle{plain} \newtheorem{theorem}{\noindent\mbox{Theorem}}
\theoremstyle{plain}\newtheorem{proposition}{\noindent\mbox{Proposition}}
\theoremstyle{plain}
\theoremstyle{definition} \newtheorem{definition}{\noindent\mbox{Definition}}
\theoremstyle{definition}
\theoremstyle{plain} 
 \def\proof{\noindent{\it Proof.~~}}
 \def\qed{\hfill$\Box$\medskip}

\begin{document}
 \title{{  Biased Random Walk on $\mathbb Z_+$ with Traps of Linearly Increasing Depth}\footnote{Supported by Nature Science Foundation of Anhui Educational Committee, Anhui, China (Grant No.
2023AH040025).}}

\author{Hua-Ming \uppercase{Wang}$^{\dag,\ddag}$ and Ning  \uppercase{Wang}$^{\dag,\S}$  }
\date{}
\maketitle%
 \footnotetext[2]{School of Mathematics and Statistics, Anhui Normal University, Wuhu 241003, China}
\footnotetext[3]{Email: hmking@ahnu.edu.cn}
\footnotetext[4]{Email: 3027912956@qq.com}

\vspace{-.5cm}

\begin{center}
\begin{minipage}[c]{12cm}
\begin{center}\textbf{Abstract}\quad \end{center}
%

 We study a $\lambda$-biased random walk $(X_n)_{n\ge0}$ on the deterministic infinite rooted tree
$\mathcal{T}=\{(i,j): i\ge0,\,0\le j\le i\}$, whose backbone is $\{(i,0):i\ge0\}$ and, for each $i\ge1$, the segment $\{(i,j):1\le j\le i\}$ forms a \emph{trap} attached to $(i,0)$. The trapping effect induces long sojourns, yielding asymptotics markedly different from simple random walks.
The walk is recurrent for $\lambda\ge1$ and transient for $0<\lambda<1$. In the transient regime it is sub-ballistic: its distance from the root grows logarithmically, with
\[
\liminf_{n\to\infty}\frac{|X_n|}{\log n}=\frac{1}{\log(1/\lambda)},\quad
\limsup_{n\to\infty}\frac{|X_n|}{\log n}=\frac{2}{\log(1/\lambda)},\quad\text{a.s.}.
\]

A contrast between spatial and temporal regeneration emerges. Let $C(n)$ be the number of cutpoints among the first $n$ backbone vertices and $M(N)$ the number of cut times up to time $N$. Then
\[
\lim_{n\to\infty}\frac{C(n)}{n}= 1-\lambda,\qquad \lim_{N\to\infty}\frac{M(N)}{\log N}=\frac{1-\lambda}{\log(1/\lambda)},\quad\text{a.s.},
\]
so cutpoints have positive linear density while cut times grow only logarithmically.

\vspace{0.2cm}

\textbf{Keywords:}\ Random walk with traps, Sub-ballistic behavior, Cutpoints, Cut times

\vspace{0.2cm}

\textbf{MSC 2020:}\ 60G50, 60J10, 60F15
\end{minipage}
\end{center}

 \section{Introduction}
\subsection{Background and motivation}

For any transient random walk on $\mathbb{Z}$ with stationary independent increments, the law of large numbers forces a positive speed. However, when the walk is spatially inhomogeneous or placed in a random environment, new phenomena arise: it may be transient yet have zero speed. Such walks are commonly said to be sub-ballistic.  For random walks in random environments, Solomon \cite{Sol75} showed that sub-ballistic behavior does occur. For further studies of sub-ballistic random walks in random environments, we refer to Kesten, Kozlov and Spitzer \cite{kks75}, Sznitman \cite{Sznitman04}, Zeitouni \cite{Zeitouni04} and references therein.

In the spatially inhomogeneous setting, Harris \cite{Harris52} initiated the study of near-critical random walks on $\mathbb{Z}$, providing recurrence/transience criteria. Later, Lamperti \cite{Lamperti1960,lam62} extended the analysis to more general near-critical processes (not necessarily Markov chains), giving similar criteria as well as a central limit theorem. Problems concerning processes with asymptotically zero drift are often referred to as Lamperti's problem; see Denisov, Korshunov and Wachtel \cite{dkw25}, and Menshikov, Popov and Wade \cite{mpw17}. For nearest-neighbor spatially inhomogeneous random walks on $\mathbb{Z}_+$ where the drift at $i$ decays as $\frac{c}{i^\alpha}$ with $\alpha\in(0,1)$ and $c>0$, Voit \cite{v92} proved that the walk is sub-ballistic and grows like $c(1+\alpha)n^{1/(1+\alpha)}$. Additional examples of birth--death chains with zero speed can be found in Cs\'aki, F\"oldes and R\'ev\'esz \cite{Csaki2010} and James, Lyons and Peres \cite{James2008} et al. Finally, we recall that symmetric simple random walks on $\mathbb{Z}$ or $\mathbb{Z}^2$ are recurrent, whereas on $\mathbb{Z}^d$ ($d\ge 3$) they are transient but sub-ballistic (see, e.g., Durrett \cite[Chap. 5]{dur19}).

In this paper, we study a spatially homogeneous $\lambda$-biased random walk on $\mathbb{Z}_+$, where each vertex $k \geq 0$ is equipped with a trap of depth $k$. Although the transition mechanism is homogeneous, the walk spends substantial time exploring traps, leading to asymptotic behavior markedly different from that of simple random walks.
We show that in the transient regime ($\lambda < 1$), the walk is sub-ballistic and its position grows only logarithmically in time.

A concept intimately related to the speed is that of cutpoints and cut times, which reflect the spatial and temporal regeneration of the walk, respectively.

\begin{definition}\label{def:cut}
  Let $(X_n)_{n\ge0}$ be a Markov chain on a countable state space. For $n \ge 1$, if
  \begin{align*}
    \{X_0,\dots,X_{n-1}\}\cap \{X_n,X_{n+1},\dots\}=\emptyset,
  \end{align*}
  then $n$ is called a \emph{cut time} and $X_n$ is called a \emph{cutpoint} of the chain.
\end{definition}

These notions have been extensively studied. Consider simple random walk on $\mathbb{Z}^d$ and let $R_n$ be the number of cut times before time $n$. For $d\ge 5$, Erd\H{o}s and Taylor \cite{Erdos1960} proved that $R_n\sim c_d n$ a.s., while for $d=3,4$, Lawler \cite{Lawler1991,Lawler1996} showed that $\log R_n\sim c_d\log n$ almost surely, where $c_d$ are positive constants.

James and Peres \cite{James1996} proved that every transient random walk with bounded step size on an integer lattice has infinitely many cutpoints almost surely. They further showed that for any Markov chain that has infinitely many cutpoints almost surely, the eventual occupation numbers generate the exchangeable $\sigma$-field, partially resolving a conjecture of Diaconis and Freedman \cite{df}. However, James, Lyons and Peres \cite{James2008} gave an example of a transient birth--death chain that has only finitely many cutpoints. Cs\'aki, F\"oldes and R\'ev\'esz \cite{Csaki2010} provided a criterion for determining whether the number of cutpoints (or strong cutpoints) is finite or infinite. Later, Benjamini, Gurel-Gurevich and Schramm \cite{Benjamini2011} proved that every transient Markov chain has infinite expected number of cutpoints, but examples of chains with only finitely many cutpoints show that the expectation can be infinite while the actual number is finite. Recently, Halberstam and Hutchcroft \cite{Halberstam2023} gave a condition in terms of Green's function decay that guarantees infinitely many cut times for general Markov chains.

For more general adapted processes with bounded increments, Lo, Menshikov and Wade \cite{lmw} gave conditions to determine whether the number of cutpoints is finite or not, and further provided lower bounds for the expected number of cutpoints in $[0,n]$ for sufficiently large $n$. For a spatially inhomogeneous nearest-neighbor random walk on $\mathbb{Z}_+$, Wang and Wang \cite{ww25} showed that if the local drift at $n$ decays like $\frac{\Upsilon}{2n}$ with a constant $\Upsilon>1$, then the number of cutpoints in $[0,n]$, normalized by $\log n$, converges in distribution to a $\Gamma(\Upsilon-1,1)$ random variable.

In our $\lambda$-biased random walk on $\mathbb{Z}_+$ equipped with traps, the walk spends most of its time wandering inside the traps, so its advance along $\mathbb{Z}_+$ is extremely slow. Indeed, we show that the distance from the root grows as $\log n$. For cut times and cutpoints, our model exhibits novel limit behaviors. We prove that the number of cut times up to time $N$ grows only logarithmically, whereas the number of cutpoints among $\{0,1,\dots,n\}$ grows linearly with $n$. This reflects the fact that the walk rarely experiences genuine regeneration events because most of its time is consumed in the traps.

Before proceeding further, we introduce some notation used throughout the paper.
\paragraph{Notation.}
We write $a_n\sim b_n$ if $a_n/b_n\to1$ as $n\to\infty$. The notation $a_n=O(b_n)$ means that $|a_n|\le c b_n$ for some constant $c$ and all sufficiently large $n$; similarly, $a_n=o(b_n)$ indicates $\lim_{n\to\infty}a_n/b_n=0$. Regarding random variables $X$ and $Y,$ when writing $X\overset{d}{=}Y,$ we mean $X=Y$ in distribution. Throughout this work, constants depending only on $\lambda$ are denoted by $c_i$ and $K_i$ ($i=0,1,2,\dots$). The values of $c_i$ are fixed once and for all, whereas the values of $K_i$ may vary from one occurrence to another.  Furthermore, the indicator of an event $A$ is written as $\mathbf{1}_A$. Finally, for a set $S$, $|S|$ denotes its cardinality.

\subsection{Model and main results}
We construct an infinite rooted tree $\mathcal T$ as follows.
The vertex set is
\[
S = \{(i,j): i = 0,1,2,\dots,\; j = 0,1,\dots,i\}.
\]
We call $(0,0)$ the {\it root}.
The set $\{(i,0): i\ge 0\}$ is called the {\it backbone} of $\mathcal T$.
For each $i\ge 1$, the set $\{(i,j): 1\le j\le i\}$ is a {\it bush} attached to the backbone vertex $(i,0)$; the vertex $(i,i)$ is called a {\it leaf}.
The tree structure of $\mathcal T$ is given by the following parent-child relations:
\begin{itemize}
  \item For $i\ge 1$, the backbone vertex $(i,0)$ has parent $(i-1,0)$ and two children: the next backbone vertex $(i+1,0)$  and the first bush vertex $(i,1).$
  \item For $i\ge 2$ and $1\le j<i$, the bush vertex $(i,j)$ has parent $(i,j-1)$ and child $(i,j+1)$.
  \item For $i\ge1,$ the leaf $(i,i)$  has parent $(i,i-1)$ and no children.
\end{itemize}

Fix $\lambda>0$.  The $\lambda$-biased random walk $(X_n)_{n\ge0}$ on $\mathcal T$ is defined by the following transition probabilities: for any $n\ge 0,$
\begin{itemize}
  \item  $P\bigl(X_{n+1}=(1,0)\mid X_n=(0,0)\bigr)=1;$
  \item for any other vertex $x\ne (0,0)$ with $k\ge0$  children,
    \begin{gather*}
    P\bigl(X_{n+1}=\text{parent of }x\mid X_n=x\bigr)=\frac{\lambda}{k+\lambda},\\
        P\bigl(X_{n+1}=\text{each child of }x\mid X_n=x\bigr)=\frac{1}{k+\lambda},
    \end{gather*}
\end{itemize}
where for $k=0,$ since $x$ has no child, it should be understood that for the next step, the walk moves to its parent with probability one. In what follows, unless otherwise stated, we always assume that $X_0=(0,0).$

\begin{remark}
  The process $(X_n)_{n\ge0}$ can be interpreted as a random walk on the non-negative integers $\mathbb Z_+$ where each integer $i\ge1$ is equipped with a trap of depth $i$.  When the walk is at $i\ge1$, it falls into the trap with probability $\frac{1}{\lambda+2}.$ Once inside the trap, it may spend a long time before exiting.  Consequently, the walk may advance very slowly along $\mathbb Z_+$.
\end{remark}

Throughout, for a vertex $(i,j)$ in the tree, its distance from the root is defined by $|(i,j)|=i+j$.
Our first theorem determines when the walk is transient or recurrent, and proves that in the transient regime it has zero speed.

\begin{theorem}\label{thm:zero}
If $0<\lambda<1$, then $(X_n)_{n\ge0}$ is sub-ballistic, i.e., $$\lim_{n\to\infty}|X_n|= \infty\quad\text{but}\quad \lim_{n\to\infty}\frac{|X_n|}{n}=0, \quad\text{a.s..} $$
If $\lambda\ge 1,$ then
$$0= \liminf_{n\to\infty}|X_n|<\limsup_{n\to\infty}|X_n|=\infty,\quad\text{a.s.,}$$ that is, the walk is recurrent.
\end{theorem}

Next, we consider the number of backbone vertices visited by time $n.$ For $k\ge 0$ let $$ S_k=\min\{n\ge 0: X_n=(k,0)\}$$ be the first time the walk hits the backbone vertex $(k,0)$.
For $n\ge 0,$ define
\[
N_n = \max\{k \ge 1 :  S_k \le n\},
\] which records the number of backbone vertices visited up to time $n.$
The next proposition shows that $N_n$  grows logarithmically.

\begin{proposition}\label{thm:logscale} For $0<\lambda\ne 1,$ we have
\begin{align*}
\lim_{n\to\infty}\frac{N_n}{\log n}
=\frac{1}{\log(\lambda\vee\lambda^{-1})},
\quad\text{a.s.}.
\end{align*}
\end{proposition}
For $n\ge 0,$ write $X_n = (X_{n,1}, X_{n,2})$, where $X_{n,1} \in \mathbb{Z}$ is current  backbone coordinate and $X_{n,2}$ is the depth in the current bush. Clearly, $0\le X_{n,2}\le X_{n,1}.$
Based on Proposition~\ref{thm:logscale}, we obtain the following asymptotic growth rates for the walk.
\begin{theorem}\label{cor:limsup}

If $0<\lambda<1$, then \begin{align}\label{eq:limb}
  \lim_{n\to\infty}\frac{X_{n,1}}{\log n}=\frac{1}{\log(1/\lambda)},\quad \text{ a.s.,}
\end{align}
and
\begin{align}\label{eq:lslim}
\liminf_{n\to\infty}\frac{|X_n|}{\log n}=\frac{1}{\log(1/\lambda)}, \quad
\limsup_{n\to\infty} \frac{|X_n|}{\log n} = \frac{2}{\log(1/\lambda)}, \quad \text{a.s.}.\end{align}

If $\lambda>1$, then
\[
\limsup_{n\to\infty}\frac{X_{n,1}}{\log n}=\frac{1}{\log\lambda},\quad\text{a.s.}.\]
\end{theorem}
\begin{remark}
  For $\lambda>1$, in the traps, the walk has a drift $\frac{\lambda-1}{\lambda+1}$ towards the backbone. Consequently, whenever it enters a trap, it is unlikely to penetrate deeply. Based on this observation, we conjecture that
  \[
  \limsup_{n\to\infty}\frac{|X_n|}{\log n} = \frac{1}{\log\lambda},\quad\text{a.s.},
  \]
  but we do not yet have a proof.
\end{remark}

Next we study the number of cutpoints and cut times of $(X_n)_{n\ge0}$. Note that for $i\ge j\ge 1$, after reaching a bush vertex $(i,j)$ the walk must return to the backbone vertex $(i,0)$; hence none of the bush vertices are cutpoints.
For $n, N \ge 0$, define
\begin{gather*}
C(n) = \bigl| \{ 0 < k \le n : (k,0) \text{ is a cutpoint} \} \bigr|, \\
M(N) = \bigl| \{ 0 < t \le N : t \text{ is a cut time} \} \bigr|,
\end{gather*}
which count, respectively, the number of cutpoints among the backbone vertices $(1,0),\dots,(n,0)$ and the number of cut times up to time $N$.

The following theorem highlights a striking contrast between the asymptotic behaviors of cutpoints and cut times.

\begin{theorem}\label{thm:cutpoints}
For $0<\lambda<1,$  we have
\[
\lim_{n\to\infty}\frac{C(n)}{n}=1-\lambda \quad \text{and}\quad \lim_{N\to\infty}\frac{M(N)}{\log N}=\frac{1-\lambda}{\log(1/\lambda)},\quad\text{a.s.}.
\]
\end{theorem}

This result shows that while cutpoints occur with a positive linear density, cut times are much rarer, growing only logarithmically. This sharp difference arises because the walk spends most of its time wandering in the bushes, so genuine regeneration events (cut times) happen only when the walk moves along the backbone.

\paragraph{Outline of the paper.}
The paper is organized as follows. Section \ref{sec:B} first collects auxiliary estimates used throughout the paper. Section \ref{sec:C} then studies recurrence, transience and zero speed, giving the proof of Theorem~\ref{thm:zero}. Next, Section \ref{sec:D} establishes the logarithmic scaling of backbone visits and displacement, proving Proposition~\ref{thm:logscale} and Theorem~\ref{cor:limsup}. Finally, Section \ref{sec:E} analyzes cutpoints and cut times, providing the proof of Theorem~\ref{thm:cutpoints}.

\section{Auxiliary results: exit probabilities, bush return times and hitting times}\label{sec:B}

This section collects essential estimates for later proofs.
First, we compute exit probabilities (gambler's ruin) for the walk on the backbone.
Second, we analyze the time spent inside a bush before returning to the backbone.
Finally, we study the travel times $T_k$ between consecutive backbone vertices $(k-1,0)$ and $(k,0)$, which will be crucial for establishing sub-ballistic behavior and logarithmic scaling.

\subsection{Exit probability from an interval}

For integers $0\le a\le m\le b$, let $P_m(a,b,+)$ denote the probability that starting from $(m,0)$ the walk hits $(b,0)$ before $(a,0)$.
The following lemma gives the formula of $P_m(a,b,+)$ in term of the parameter $\lambda.$
\begin{lemma}
  For $0\le a\le m\le b$,
  \begin{equation}\label{eq:pmab}
P_m(a,b,+) =
\begin{cases}
\dfrac{1 - \lambda^{m-a}}{1 - \lambda^{b-a}}, & \text{if } \lambda \neq 1,\\[10pt]
\dfrac{m-a}{b-a}, & \text{if } \lambda = 1.
\end{cases}
\end{equation}
\end{lemma}
\begin{proof}
Conditioning on the first step from $(m,0)$ gives
\begin{align*}
  P_m(a,b,+)=\frac{\lambda}{2+\lambda}P_{m-1}(a,b,+)
  +\frac{1}{2+\lambda}P_{m+1}(a,b,+)+\frac{1}{2+\lambda}P_{m}(a,b,+),
\end{align*}
for $a<m<b.$ As a consequence, we get
\begin{align}\label{pmab}
  P_m(a,b,+)=\frac{\lambda}{1+\lambda}P_{m-1}(a,b,+)
  +\frac{1}{1+\lambda}P_{m+1}(a,b,+),\quad a<m<b.
\end{align}
For the boundary values, clearly we have
\begin{align}\label{bv}
   P_a(a,b,+)=0, \quad P_b(a,b,+)=1.
\end{align}
Solving the linear recurrence \eqref{pmab} with boundary conditions \eqref{bv} gives  \eqref{eq:pmab}.\qed
\end{proof}
\begin{remark}
  Let $a\le m\le b$ be integers and $0<p<1$. Set $\rho = (1-p)/p$. Consider a nearest-neighbor random walk $(Y_n)_{n\ge0}$ on $\mathbb{Z}$ starting from $m$, which moves right with probability $p$ and left with probability $1-p$ at each step. Denote also by $P_m(a,b,+)$ the probability that the walk hits $b$ before $a$. Then the classical gambler's ruin formula gives
  \begin{align}\label{eq:gam}
    P_m(a,b,+) =
    \begin{cases}
      \dfrac{1 - \rho^{\,m-a}}{1 - \rho^{\,b-a}}, & \text{if } p \neq 1/2,\\[10pt]
      \dfrac{m-a}{b-a}, & \text{if } p = 1/2.
    \end{cases}
  \end{align}
\end{remark}
\subsection{Bush return time}
Fix $k\ge2.$ Let $R_{k-1}$ be the time (number of steps) it takes to return to $(k-1,0)$ after first entering the bush, i.e., starting from $(k-1,1)$ and stopping upon first hitting $(k-1,0)$. By the strong Markov property, $R_{k-1}$ is independent of the rest of the walk and has a distribution depending only on $k-1$. We have the following lemma.

\begin{lemma}\label{lem:bushre}
Fix $k\ge 2$. Then for $0<\lambda\ne 1,$ we have
\begin{align}\label{eq:erk1}
\mathbb{E}[R_{k-1}] =
     \frac{2}{1-\lambda}\lambda^{-(k-2)} - \frac{1+\lambda}{1-\lambda},
 \end{align}
 and there exists a constant $c_1>0$  such that for all $k\ge 2$,
\begin{align}\label{eq:er12}
\mathbb{E}[R_{k-1}^2] \le c_1 \bigl(\mathbb{E}[R_{k-1}]\bigr)^2.
\end{align}
\end{lemma}

\begin{proof}
The bush attached to $(k-1,0)$ is a path of length $k-1$ with vertices $(k-1,1),\dots,(k-1,k-1),$ where $(k-1,k-1)$ is a leaf.
Let $\xi_0=0$ and for $j=1,\dots,k-1$ let $\xi_j$ be the number of steps to hit $(k-1,0)$ starting from $(k-1,j)$.
Set $E_j = \mathbb{E}[\xi_j]$, $F_j = \mathbb{E}[\xi_j^2]$, $0\le j\le k-1$.

Consider at first the first moments.
 Conditioning on the first step gives
\begin{align*}
&E_j = 1 + \frac{\lambda}{1+\lambda}E_{j-1} + \frac{1}{1+\lambda}E_{j+1}, \quad 1\le j\le k-2,\\
&E_{k-1} = 1 + E_{k-2}.
\end{align*}
Solving these equations,  we obtain
\begin{align}\label{eq:Ej}
  E_j
=  \frac{2\lambda^{-(k-2)}}{(1-\lambda)^2}(1-\lambda^j) - \frac{1+\lambda}{1-\lambda}j,\quad 1\le j\le k-1.
\end{align}
Noticing that $R_{k-1}=E_1,$ thus putting $j=1$ in \eqref{eq:Ej}, we get \eqref{eq:erk1}.

For the second moments,
  conditioning on the first step gives
\begin{align*}
&F_j = \frac{\lambda}{1+\lambda} E[(1+\xi_{j-1})^2] + \frac{1}{1+\lambda} E[(1+\xi_{j+1})^2], \quad 1 \le j \le k-2,\\
&F_{k-1} = E[(1+\xi_{k-2})^2] = 1 + 2E_{k-2} + F_{k-2}.
\end{align*}
As a consequence, we get
\begin{align*}
&F_j = 1 + \frac{\lambda}{1+\lambda}(2E_{j-1}+F_{j-1}) + \frac{1}{1+\lambda}(2E_{j+1}+F_{j+1}), \quad 1 \le j \le k-2,\\
&F_{k-1} = 1 + 2E_{k-2} + F_{k-2}.
\end{align*} Define $\Delta_j = F_{j+1}-F_j$ for $j=0,\dots,k-2$.
We infer that
\begin{align}\label{eq:F}
&\Delta_{j-1} = \frac{1}{\lambda}\Delta_j + \frac{1+\lambda}{\lambda} + \frac{2}{\lambda}(\lambda E_{j-1} + E_{j+1}), \ \ 1\le j\le k-2,\\
\label{eq:G}
&\Delta_{k-2} = 1 + 2E_{k-2}.
\end{align}

To show the second part of the lemma, we
 treat two cases $0<\lambda<1$ and $\lambda>1$ separately.

\paragraph{\it Case $\lambda<1.$}
By \eqref{eq:Ej}, we have
\begin{align*}
E_j \le K_1 \lambda^{-k} \quad \text{for all } j=0,\dots,k-1,
\end{align*}
with a constant $K_1>0$ independent of $k.$
 Consequently, on accounting of \eqref{eq:F} and \eqref{eq:G},  for some constant $K_2>0$ depending only on $\lambda$ we have
\[
\Delta_{j-1} \le \frac{1}{\lambda} \Delta_j + K_2\lambda^{-k}, \quad 1\le j\le k-2
\]
and $\Delta_{k-2} \le K_2\lambda^{-k}$.  Iterating backward, we obtain
\[
\Delta_0 \le K_2\lambda^{-k}\sum_{i=0}^{k-2}\lambda^{-i} \le \frac{K_2}{1-\lambda}\lambda^{-2k}.
\]
  But from \eqref{eq:erk1}, we have  $\mathbb{E}[R_{k-1}]  \ge K_3\lambda^{-k}$ for some $K_3>0$ depending only on $\lambda.$ Thus with $c_1 = \frac{K_2}{K_3^2(1-\lambda)}$, we get  $$E[R_{k-1}^2]=F_1 =\Delta_0 \le c_1 (E[R_{k-1}])^2.$$

\paragraph{\it Case $\lambda>1.$}
For $\lambda>1,$ from \eqref{eq:Ej}, we get  $E_j<K_4j,$ $1\le j\le k-1$ for some constant $K_4$ independent $j,k.$
 Hence, from  \eqref{eq:F} and \eqref{eq:G},  for proper constant $K_5$ depending only on $\lambda,$ we have \[
\Delta_{j-1} \le \frac{1}{\lambda} \Delta_j + K_5j,\quad 1\le j\le k-2
\]
and $\Delta_{k-2} \le K_5(k-1)$.  Iterating, we get
\[
\Delta_0 \le K_5 \sum_{i=1}^{k-1}i\lambda^{-(i-1)} \le K_6,
\] for some constant $K_6$ depending only on $\lambda.$
 Since $\mathbb{E}[R_{k-1}]\ge 1,$  we have $$E[R_{k-1}^2]=F_1=\Delta_0\le K_6 \le  K_6(E[R_{k-1}])^2.$$ Thus, taking $c_1=K_6,$  we get \eqref{eq:er12}.
\qed
\end{proof}

\subsection{Hitting times of the backbone vertices}

Let $ S_0=0$ and for $k\ge1$ define
\[
 S_k = \inf\{n> 0: X_n = (k,0)\},
\]
the first time that the walk reaches the backbone vertex $(k,0)$. Set $$T_k =  S_k - S_{k-1}\text{ for } k\ge1.$$ By the strong Markov property, the random variables $\{T_k\}_{k\ge1}$ are independent. Since  the bush at $(k,0)$ has length $k,$ the distribution of $T_k$ varies with $k.$ The following proposition plays a key role in the proofs of Proposition \ref{thm:logscale}, Theorems \ref{thm:zero} and \ref{cor:limsup}.
\begin{proposition}\label{prop:tk} Fix $0<\lambda\ne 1.$ {\rm(i)} We have
\begin{align}\label{eq:etk}
  E[T_k]= \frac{2\lambda(\lambda^2+\lambda-1)}{(\lambda-1)^2(1+\lambda)}\lambda^{k-1}
-\frac{\lambda^2+2\lambda-1}{(\lambda-1)^2}
+\frac{2\lambda^{2-k}}{(\lambda-1)^2(1+\lambda)},\quad k\ge1,
\end{align}
 and \begin{align}\label{eq:etkb}
  E[T_k^2]\le c_2\left(\lambda\vee \lambda^{-1}\right)^{2k}, \quad k\ge 1
\end{align} for some constant $c_2$ depending only on $\lambda.$

 {\rm(ii)}
  There exist constants $c_3>0$ and $\delta>0$ depending only on $\lambda$ such that
\begin{equation}\label{eq:Tk_bound}
P(T_k \ge c_3 (\lambda\vee\lambda^{-1})^k) \ge \delta,\quad k\ge1.
\end{equation}
\end{proposition}


\proof
For simplicity, for $k\ge 1,$ we write
\[
a_k = E[T_k],\qquad a_k^{(2)} = E[T_k^2].
\]
By the first-step decomposition from $(k-1,0),$ we have
\begin{align}\label{eq:Tkdecom}
T_k =
\begin{cases}
1, & \text{if the first step is to } (k,0), \\
1 + T_{k-1}' + T_k', & \text{if the first step is to } (k-2,0) , \\
1 + R_{k-1} + T_k'', & \text{if the first step is into the bush},
\end{cases}
\end{align}
where
 $T_{k-1}'$ denotes the number of steps of the walk, starting from $(k-2,0),$ ending at $(k-1,0)$, which has the same distribution as $T_{k-1};$ $T_k'$ denotes the number of steps of the walk, starting from $(k-1,0),$ ending at $(k,0),$ which has the same distribution as $T_k$ and is independent of $T_{k-1}';$ $T_k''$ denotes also the number of steps of the walk, starting from $(k-1,0),$ ending at $(k,0),$ which has the same distribution as $T_k$ and is independent of $R_{k-1}.$

From  \eqref{eq:Tkdecom}, we deduce that
\begin{align*}
a_k& = \frac{1}{2+\lambda}  + \frac{\lambda}{2+\lambda} E[1 + T_{k-1}' + T_k'] + \frac{1}{2+\lambda} E[1 + R_{k-1} + T_k'']
\end{align*}
and
\begin{align*}
  a_k^{(2)}& = \frac{1}{2+\lambda} + \frac{\lambda}{2+\lambda} E[(1 + T_{k-1}' + T_k')^2] + \frac{1}{2+\lambda} E[(1 + R_{k-1} + T_k'')^2].
\end{align*}
Using independence and the facts $T_k''\overset{d}{=}T_k,$  $T_k'\overset{d}{=}T_k$ and $T_{k-1}'\overset{d}{=}T_{k-1}$, we get
\begin{equation}\label{eq:ak_sol}
a_k = 2+\lambda + \lambda a_{k-1} + E[R_{k-1}]
\end{equation}
and
\begin{align}
  \label{eq:ak2}
a_k^{(2)}& = 2+\lambda + \lambda \left(a_{k-1}^{(2)} + 2a_{k-1} + 2a_k + 2a_{k-1}a_k\right)\nonumber\\
 &+  \left(E[R_{k-1}^2] + 2E[R_{k-1}] + 2a_k + 2E[R_{k-1}]a_k\right).
\end{align}
Notice that by Lemma \ref{lem:bushre}, for $\lambda\ne 1$ and $k\ge 2,$ we have
\begin{align}\label{eq:erkln1}
\mathbb{E}[R_{k-1}] =
      \frac{2}{1-\lambda}\lambda^{-(k-2)} - \frac{1+\lambda}{1-\lambda}.
\end{align}
Substituting \eqref{eq:erkln1} into \eqref{eq:ak_sol}, we get
\begin{equation}\label{eq:akk1}
a_k =  \lambda a_{k-1} + 2+\lambda+\frac{2}{1-\lambda}\lambda^{-(k-2)} - \frac{1+\lambda}{1-\lambda}.
\end{equation}
Noticing that $a_1=1,$ thus iterating  \eqref{eq:akk1}, gives \eqref{eq:etk}.

Next we estimate the second moment.  Fix $0<\lambda\ne 1.$ In view of \eqref{eq:etk}, we can find a constant $K_1$ depending only on $\lambda$ such that
\begin{align*}
  a_k=E[T_k]\le K_1\left(\lambda\vee \lambda^{-1}\right)^k,\quad k\ge1.
\end{align*}
Furthermore, by Lemma \ref{lem:bushre}, there exists a constant $K_2$ depending only on $\lambda$ such that
\begin{align*}
  E[R_{k-1}]\le K_2 \left(\lambda\vee \lambda^{-1}\right)^k, \quad E[R_{k-1}^2]\le K_2 \left(\lambda\vee \lambda^{-1}\right)^{2k},\quad k\ge2.
\end{align*}
Therefore, on accounting of \eqref{eq:ak2}, we can find a proper constant $K_3$ depending only on $\lambda$ such that
\begin{align}\label{eq:ak2k1}
  a_k^{(2)} \le   \lambda a_{k-1}^{(2)}+K_3\left(\lambda\vee \lambda^{-1}\right)^{2k}.
\end{align}
We now show by induction that there exists a constant $c_2$ such that \begin{align}\label{eq:aktb} a_k^{(2)} \le c_2 \left(\lambda\vee \lambda^{-1}\right)^{2k},\quad k\ge1.\end{align}
Since $a_1^{(2)}=1,$ we can choose $c_2$ large enough so that $a_1^{(2)} \le c_2 \left(\lambda\vee \lambda^{-1}\right)^2$.
Assume that $a_{k-1}^{(2)} \le c_2 \left(\lambda\vee \lambda^{-1}\right)^{2(k-1)}$. Then by \eqref{eq:ak2k1}, we deduce that
\begin{align*}a_k^{(2)} &\le   \lambda c_2 \left(\lambda\vee \lambda^{-1}\right)^{2(k-1)}+K_3\left(\lambda\vee \lambda^{-1}\right)^{2k}\\
&=\left(\lambda\left(\lambda\vee \lambda^{-1}\right)^{-2}c_2+K_3\right)\left(\lambda\vee \lambda^{-1}\right)^{2k}.\end{align*}
Since $\lambda\left(\lambda\vee \lambda^{-1}\right)^{-2}<1,$ choosing $c_2$ large enough so that $\lambda\left(\lambda\vee \lambda^{-1}\right)^{-2}c_2+K_3<c_2,$ we get $a_k^{(2)}\le c_2\left(\lambda\vee \lambda^{-1}\right)^{2k}.$ Thus \eqref{eq:aktb} is true and so is \eqref{eq:etkb}.
Part (i) of Proposition \ref{prop:tk} is proved.

To prove Part (ii) of Proposition \ref{prop:tk}, we use the Paley-Zygmund inequality: for any non-negative random variable $X$ with finite second moment,
\begin{align}\label{eq:PZ}
P\left(X\ge\frac12E[X]\right)\ge\frac14\frac{(E[X])^2}{E[X^2]}.
\end{align}
Notice that by \eqref{eq:etk} and \eqref{eq:etkb}  we have
$$E[T_{k}] \ge K_4 (\lambda\vee\lambda^{-1})^{k}, \quad E[T_{k}^2] \le c_2 (\lambda\vee\lambda^{-1})^{2k}, \quad k\ge1$$ for some constant $K_4>0$ and $c_2>0$ which depend only on $\lambda.$
Taking $X=T_k,$ from \eqref{eq:PZ}   we get
\[
P\left(T_k\ge\frac{K_4}{2}(\lambda\vee\lambda^{-1})^{k}\right)\ge P\left(T_k\ge\frac{1}{2}E[T_k]\right)\ge\frac{K_4}{4c_2}, \quad k\ge1.
\]
Thus, setting $c_3=\frac{K_4}{2},$ $\delta=\frac{K_4}{4c_2},$ we obtain \eqref{eq:Tk_bound}.
 This completes the proof of Part (ii) of Proposition \ref{prop:tk}. \qed

\section{Recurrence criterion and sub-ballistic behavior: proof of Theorem \ref{thm:zero}}\label{sec:C}

\proof For $0\le a\le m\le b,$
letting $b\to\infty$ in \eqref{eq:pmab}, we obtain
\begin{equation}\label{eq:pmainf}
P_m(a,\infty,+) =\begin{cases} 1 - \lambda^{m-a}, &\text{if }0<\lambda<1,\\
0, &\text{if }\lambda\ge 1,\end{cases}
\end{equation}
the probability the walk never hits $(a,0),$ starting from $(m,0).$
In particular, starting from $(m+1,0)$, the probability the walk  never returns to $(m,0)$ is
\[
P_{m+1}(m,\infty,+) =\begin{cases} 1 - \lambda, &\text{if }0<\lambda<1,\\
0,& \text{if }\lambda\ge 1.\end{cases}
\]
If $0<\lambda<1$, from each backbone vertex there is a positive chance to never return to the previous one. Hence the walk is transient. If $\lambda\ge1,$ from each backbone vertex, with probability $1$ the walk will return to the previous one. Hence the walk is recurrent.

Now we show that for the transient case, $(X_n)_{n\ge0}$ is sub-ballistic. Indeed, this is a corollary of Theorem \ref{cor:limsup}. But here we prove it directly, since the proof is not too long. Suppose $0<\lambda<1.$ Notice that
 the total time to reach $(n,0)$ is $S_n = \sum_{k=1}^n T_k$. Clearly $|X_{S_n}| = n$. For any $m$, let $n$ be the unique integer such that $S_n \le m < S_{n+1}$. Then the possible maximal distance of $X_m$ from the root is $2n$ (attained at the leaf $(n,n)$). Hence $|X_m| \le 2n$, and therefore
\[
\frac{|X_m|}{m} \le \frac{2n}{S_n}.
\]
If we prove that \begin{align}\label{snni}\lim_{n\to\infty}S_n/n =\infty, \quad \text{a.s.,}\end{align}  then $|X_m|/m \to 0$ almost surely as $m\to\infty,$  that is,  the walk has zero speed.

Next we show \eqref{snni}.
Since $T_k,$ $k\ge1$ are independent,  in view of \eqref{eq:Tk_bound}, the second Borel-Cantelli lemma then implies that almost surely, $T_k \ge c_3 \lambda^{-k}$ for infinitely many $k$. Let $k_1 < k_2 < \cdots$ be those indices. Then for any $n$, let $$k(n) = \max\{k_j : j\ge 1, k_j \le n\}.$$
Clearly, we have $S_n \ge T_{k(n)} \ge c_3 \lambda^{-k(n)}$.  Let $A_k = \{T_k \ge c_3 \lambda^{-k}\}$. By Kolmogorov's strong law of large numbers for independent random variables (see e.g., Theorem 2 in \S3, Chapter IV of Shiryaev \cite{s}), we have
\[
\frac1n\sum_{k=1}^n\mathbf{1}_{A_k} - \frac1n\sum_{k=1}^n P(A_k) \to 0, \quad \text{a.s..}
\]
From \eqref{eq:Tk_bound}, we see that $\liminf_{n\to\infty}\frac1n\sum_{k=1}^n P(A_k) \ge \delta.$ We thus have $\liminf_{n\to\infty} \frac1n\sum_{k=1}^n \mathbf{1}_{A_k} \ge \delta$ almost surely. Consequently, for  $n$ large enough, we have
\[
M_n := \#\{k \le n : A_k \text{ occurs}\} \ge  \delta n/2, \quad \text{a.s.}.
\]
 Hence, we get
  $$k(n) \ge M_n\ge \delta n/2,\quad \text{a.s.,}$$  for large $n.$   As a result,
\[
\frac{S_n}{n} \ge \frac{T_{k(n)}}{n} \ge \frac{c_3 \lambda^{-k(n)}}{n} \ge \frac{c_3 \lambda^{- \delta n/2}}{n} ,
\]
almost surely for large $n.$ Therefore, we get \eqref{snni}. Theorem \ref{thm:zero} is proved. \qed

\section{Logarithmic scaling of the backbone visits and displacement}\label{sec:D}

In this section,  we prove Proposition~\ref{thm:logscale} and Theorem~\ref{cor:limsup}. The proof of Proposition \ref{thm:logscale} is given in Subsection \ref{sec:41}. The proof of Theorem~\ref{cor:limsup} is longer and is divided into two parts:
the case $0<\lambda<1$ is treated in Subsection~\ref{sec:42},
and the case $\lambda>1$ in Subsection~\ref{sec:43}.

\subsection{Proof of Proposition \ref{thm:logscale}}\label{sec:41}

Recall that $S_n = \sum_{k=1}^n T_k$ is the time that the walk takes to reach $(n,0)$. To prove Proposition~\ref{thm:logscale}, we need the following lemma, which establishes the exponential growth rate of $S_n$.
\begin{lemma}\label{lem:logSn}
For $0<\lambda\ne 1,$ we have
\[
\lim_{n\to\infty}\frac{\log S_n}{n}=\log\left(\lambda\vee \lambda^{-1}\right),\quad\text{a.s.}.
\]
\end{lemma}

\begin{proof}
For $k\ge1,$ define $U_k=\left(\lambda\vee \lambda^{-1}\right)^{-k} T_k$. Then $U_k,k\ge1$ are independent. By Proposition \ref{prop:tk},
 there exist constants  $K_1$ such that \begin{align}\label{eq:evuk}E[U_k] \le K_1\text{ and } \text{Var}(U_k) \le K_1, \quad k\ge1.
\end{align}
Consider $$V_n=\left(\lambda\vee \lambda^{-1}\right)^{-n} S_n=\sum_{k=1}^n\left(\lambda\vee \lambda^{-1}\right)^{-(n-k)}U_k.$$ We show that $\frac1n\log V_n\to0$ almost surely as $n\to\infty$, which is equivalent to the desired statement.

From \eqref{eq:evuk}, we have $$E[V_n]\le K_1\sum_{j=0}^{n-1}\left(\lambda\vee \lambda^{-1}\right)^{-j}\le K_1/(1-\lambda\wedge \lambda^{-1}),\quad n\ge1$$ and
\begin{align*}
\text{Var}(V_n)&=\sum_{k=1}^n\left(\lambda\vee \lambda^{-1}\right)^{-2(n-k)}\text{Var}(U_k)\\
&\le K_1\sum_{j=0}^{n-1}\left(\lambda\vee \lambda^{-1}\right)^{-2j}\le K_1/(1-\lambda^2\wedge \lambda^{-2}),\quad n\ge1.
\end{align*}
Thus,
by Markov's inequality, for $n\ge1,$ $$P(V_n>n)\le E[V_n^2]/n^2\le  n^{-2}K_1/(1-\lambda^2\wedge \lambda^{-2}),$$ which is summable. By the Borel-Cantelli lemma, $V_n\le n$ for all sufficiently large $n$ almost surely. Hence $$\limsup_{n\to\infty}\frac1n\log V_n\le\lim_{n\to\infty}\frac{\log n}{n}=0,\quad\text{a.s..}$$

We derive next a lower bound.
Let $c_3$ and $\delta$ be the constants from \eqref{eq:Tk_bound}. Define $B_k = \{U_k \ge c_3\}$ for $k\ge1.$  Then by \eqref{eq:Tk_bound}, $B_k,$ $k\ge 1$ are independent events with $P(B_k) \ge \delta$ for  $k\ge1.$  Hence $\sum_{k=1}^\infty P(B_k) = \infty.$ The second Borel-Cantelli lemma implies that $B_k$ occurs infinitely often almost surely. Let $K_2 = \min\{k \ge 1 : B_k \text{ occurs}\}.$ Then $K_2$ is almost surely finite. For any $n \ge K_2$, the set $\{k \le n : B_k \text{ occurs}\}$ is non-empty, so we can define
\[
k_{\max}(n) = \max\{k \le n : B_k \text{ occurs}\},\qquad L_n = n - k_{\max}(n).
\]
For such $n$ we have
\begin{align*}
V_n &= \sum_{k=1}^n \left(\lambda\vee \lambda^{-1}\right)^{-(n-k)}U_k \\
&\ge \left(\lambda\vee \lambda^{-1}\right)^{-(n - k_{\max}(n))} U_{k_{\max}(n)} \ge c_3 \left(\lambda\vee \lambda^{-1}\right)^{-L_n}.
\end{align*}
Thus the inequality $V_n \ge c_3 \left(\lambda\vee \lambda^{-1}\right)^{-L_n}$ holds for all $n \ge K_2$ almost surely.

For any $\alpha>0$, the event $\{L_n > \alpha \log n\}$ means that no $B_k$ occurs for $k$ in $(n-\alpha\log n, n]$. Hence
\[
P(L_n > \alpha \log n) \le (1-\delta)^{\lfloor \alpha \log n \rfloor} \le (1-\delta)^{\alpha \log n} = n^{-\alpha \log(1/(1-\delta))}.
\]
Choose $\alpha$ large enough so that $\alpha \log(1/(1-\delta)) > 1$. Then $\sum_{n=1}^\infty P(L_n > \alpha \log n) < \infty$, and by the Borel-Cantelli lemma, almost surely $L_n \le \alpha \log n$ for all sufficiently large $n$. Consequently, for those $n$,
\[
V_n \ge c_3 \left(\lambda\vee \lambda^{-1}\right)^{-L_n} \ge c_3 \left(\lambda\vee \lambda^{-1}\right)^{-\alpha \log n} = c_3 n^{-\alpha \log\left(\lambda\vee \lambda^{-1}\right)},
\]
and therefore
\[
\liminf_{n\to\infty}\frac{\log V_n}{n} \ge \lim_{n\to\infty}\frac{\log c_3}{n} - \alpha \log\left(\lambda\vee \lambda^{-1}\right) \frac{\log n}{n}= 0
\]
almost surely. We
thus conclude that $\lim_{n\to\infty}\frac1n\log V_n =0$ almost surely. Consequently, we have $$\frac1n\log S_n = \frac1n\log V_n + \log\left(\lambda\vee \lambda^{-1}\right) \to \log\left(\lambda\vee \lambda^{-1}\right)$$ almost surely as $n\to\infty.$ \qed
\end{proof}

We are now ready to give the proof of Proposition \ref{thm:logscale}.

\noindent{\it Proof of Proposition \ref{thm:logscale}.}
Recall that for any $n\ge1$, $N_n = \max\{k \ge 1 : S_k \le n\}$. Then we have $S_{N_n} \le n < S_{N_n+1}$ which implies
\[
\frac{\log S_{N_n}}{N_n} \le \frac{\log n}{N_n} < \frac{\log S_{N_n+1}}{N_n}.
\]
 From Lemma~\ref{lem:logSn}, we have
\[
\frac{\log S_{N_n}}{N_n} \to \log\left(\lambda\vee \lambda^{-1}\right)\quad\text{and}\quad \frac{\log S_{N_n+1}}{N_n+1} \to \log\left(\lambda\vee \lambda^{-1}\right), \quad \text{a.s.,}
\] as $n\to\infty.$
Hence
\[
\frac{\log n}{N_n} \to \log\left(\lambda\vee \lambda^{-1}\right) \quad \text{a.s.,}  
\]
as  $n\to\infty.$
This completes the proof of Proposition~\ref{thm:logscale}. \qed

\subsection{Proof of Theorem \ref{cor:limsup}: Case $0<\lambda<1$}\label{sec:42}

To prove the case $0<\lambda<1$ Theorem \ref{cor:limsup}, we need the following lemma.
\begin{lemma}\label{lem:bush_time} Suppose $0<\lambda<1.$
Let $H_k$ be the number of steps it takes for the walk, starting from $(k,1)$, to reach the leaf $(k,k)$ without ever returning to $(k,0)$. Then there exists a constant $c_4$ such that $E[H_k] \le c_4 k$ for all $k\ge 1$, and moreover, $H_k = O(k)$ almost surely.
\end{lemma}

\begin{proof}
Identify the set $\{(k,0),(k,1),\dots,(k,k)\}$ with $\{0,1,\dots,k\}$, where $0$ corresponds to the backbone vertex $(k,0)$ and $k$ to the leaf. Consider a random walk $(Y_n)_{n\ge 0}$ on $\{0,\dots,k\}$ with $P(Y_0=1)=1$ and transition probabilities
\begin{align*}
&P(Y_{n+1}=j+1\mid Y_n=j)=\frac{1}{1+\lambda}=:p,\\
&P(Y_{n+1}=j-1\mid Y_n=j)=\frac{\lambda}{1+\lambda}=:q,
\end{align*}
for $1\le j\le k-1$ and $n\ge0,$ with absorbing boundaries at $0$ and $k.$
 Let $h(j):=P_{j}(0,k,+)$ be the probability that the walk $(Y_n)_{n\ge0}$  hits $k$ before $0,$ starting from $j.$ Then by \eqref{eq:gam}, we have \begin{align}\label{eq:hj}h(j)=\frac{1-\lambda^j}{1-\lambda^k}, \quad 1\le j\le k-1.\end{align}

Let $A_k$ be the event that the walk reaches $k$ before ever returning to $0$.
We work under the conditional probability $\tilde P(\cdot)=P(\cdot\mid A_k)$. Let
$(\tilde Y_n)_{n\ge0}$ be the conditioned walk induced by $\tilde P$ with $\tilde Y_0 = 1$. Then
the transition probabilities of $(\tilde Y_n)_{n\ge0}$ under $\tilde P$ are given by the Doob $h$-transform using the harmonic function $h(j)$ defined in \eqref{eq:hj}. For $2\le j\le k-1$ and $n\ge0$, we have
\begin{align*}
&\tilde P(j,j+1):=\tilde P\left(\tilde Y_{n+1}=j+1\,\middle|\,\tilde Y_n=j\right)=\frac{p\,h(j+1)}{h(j)},\\
&\tilde P(j,j-1):=\tilde P\left(\tilde Y_{n+1}=j-1\,\middle|\,\tilde Y_n=j\right)=\frac{q\,h(j-1)}{h(j)},
\end{align*}
and for $j=1,$ since $h(0)=0$ we have $$\tilde P(1,2):=\tilde P\left(\tilde Y_{n+1}=2\,\middle|\,\tilde Y_n=1\right)=1,\ \  \tilde P(1,0):=\tilde P\left(\tilde Y_{n+1}=0\,\middle|\,\tilde Y_n=1\right)=0.$$ Note that $\tilde P(j,j+1) > p$ for all $1\le j\le k-1$ since $h$ is strictly increasing.

Let $\tau_1=0.$
Define for $j=2,\dots,k$ the hitting time
\[
\tau_j = \min\{n\ge 0: \tilde Y_n = j\},
\] and set
$\xi_j=\tau_j-\tau_{j-1}.$
Then, by the strong Markov property, the random variables $\xi_2,\dots,\xi_k$ are independent. Moreover, $H_k = \sum_{j=2}^k \xi_j.$

Next we show that for some $c_4$ depending only on $\lambda,$
\begin{align}\label{eq:yj12}
  E[\xi_j]\le c_4, \quad E[\xi_j^2]\le c_4, \quad j\ge 2.
\end{align}
To this end, we couple the condition walk with a simple random walk on $\mathbb Z.$

Fix $2\le j\le k-1.$
Suppose that $(U_n)_{n\ge 1}$ is a sequence of independent uniform $[0,1]$ random variables.
Set $Z_0=\tilde Z_0=j-1.$ For $n\ge 1,$ define recursively
\[
Z_n = \begin{cases}
Z_{n-1}+1, & \text{if } U_n < p,\\
Z_{n-1}-1, & \text{if } U_n \ge p,
\end{cases} \quad \tilde Z_n = \begin{cases}
\tilde Z_{n-1}+1, & \text{if } U_n < \tilde p_{\tilde Z_{n-1}},\\
\tilde Z_{n-1}-1, & \text{if } U_n \ge \tilde p_{\tilde Z_{n-1}},\end{cases}
\]
where
$$\tilde p_i = \begin{cases}\tilde P(i,i+1),& 1\le i\le k-1,\\
p, &i\ge k\text{ or }i\le 0.\end{cases}$$ Since $\tilde p_i \ge p$ for every $i\in \mathbb Z$, a simple induction shows that $\tilde Z_n \ge Z_n$ for all $n$ almost surely. Indeed, the claim holds for $n=0$. For all $n\ge1,$ assume $\tilde Z_{n-1} \ge Z_{n-1}$. Suppose the event $\{U_n < p\},$ the event $\{U_n < p_{\tilde Z_{n-1}}\}$  occurs as well. Thus then $\tilde Z_n = \tilde Z_{n-1}+1\ge Z_{n-1}+1= Z_n.$ Otherwise, suppose the event $\{U_n \ge p\}$ occurs, then $Z_n = Z_{n-1}-1.$ In this case, regardless of whether the event $\{U_n < \tilde p_{\tilde X_{n-1}}\}$ occurs or not, we have $\tilde Z_n \ge \tilde Z_{n-1}-1 \ge Z_{n-1}-1 = Z_n.$

Clearly
 $(Z_n)_{n\ge 0}$ is a simple random walk on $\mathbb{Z}$ with $Z_0 = j-1$ and transition probabilities
\[
P(Z_{n+1}=i+1\mid Z_n=i)=p,\qquad P(Z_{n+1}=i-1\mid Z_n=i)=q.
\]
It is also easy to check that $(\tilde Z_n)_{n\ge0}$ is a random walk starting from $j-1$ and shares the same transition probabilities with $(\tilde Y_n)_{n\ge 0}$ before leaving the set $\{1,...,k\}.$

Let $$\tilde \xi_j=\inf\{n> 0:\tilde Z_n=j\},\quad \eta_j=\inf\{n>0:Z_n=j\}.$$ Then we have $\tilde \xi_j\le \eta_j $ almost surely and  $\tilde \xi_j\overset{d}{=}\xi_j.$

For the simple random walk $(Z_n)_{n\ge0}$, it is well known that
\[
E[\eta_j] = \frac{1}{p-q} = \frac{1+\lambda}{1-\lambda},
\]
and the second moment is finite and given by
\[
E[\eta_j^2] = \frac{1}{p-q} + \frac{4q}{(p-q)^2} + \frac{2q}{(p-q)^3}
= \frac{1+\lambda}{1-\lambda} + \frac{4\lambda(1+\lambda)}{(1-\lambda)^2} + \frac{2\lambda(1+\lambda)^2}{(1-\lambda)^3}:=c_4.
\]
Therefore we obtain
\[
E[\xi_j]=E[\tilde \xi_j] \le E[\eta_j] <c_4,\qquad E[\xi_j^2] \le E[\eta_j^2]=c_4,
\]
which gives \eqref{eq:yj12}.

We have shown that
the random variables $\xi_2,\dots,\xi_k$ are independent under $\tilde P$ and have uniformly bounded second moments. By Kolmogorov's strong law of large numbers for independent random variables (see, e.g., Theorem 2 in \S3, Chapter IV of Shiryaev \cite{s}), we have
\[
\frac{1}{k-1}\sum_{j=2}^k (\xi_j - E[\xi_j]) \to 0, \quad\text{a.s.,}
\] as $k\to\infty.$
Since $E[\xi_j] \le c_4$,
\[
\frac{1}{k-1}\sum_{j=2}^k \xi_j \le c_4 + o(1) \quad \text{a.s.},
\]
which implies $H_k = \sum_{j=2}^k \xi_j = O(k)$ almost surely.  This completes the proof. \qed
\end{proof}

\noindent{\it Proof of Theorem \ref{cor:limsup}: Case $0<\lambda<1$.} To begin with, we derive the upper limit.
Since $X_n$ lies in the bush of some backbone vertex visited by time $n$, we have $|X_n| \le 2N_n$, and consequently
\begin{align}\label{eq:ubnxn}
\limsup_{n\to\infty} \frac{|X_n|}{\log n} \le \frac{2}{\log(1/\lambda)}, \quad \text{a.s..}
\end{align}
Next we derive the lower bound.
For each $k\ge1$, denote by $A_k$ the event that upon first arriving at $(k,0)$, the walk steps into the bush and then reaches the leaf $(k,k)$ before ever returning to $(k,0)$. From $(k,0)$, the probability to step into the bush is $\frac{1}{2+\lambda}$. Starting from  $(k,1)$, by \eqref{eq:gam}, the probability to hit $(k,k)$ before returning to $(k,0)$ is  $\frac{1-\lambda}{1-\lambda^k}$. Hence
\[
P(A_k) = \frac{1}{2+\lambda} \cdot \frac{1-\lambda}{1-\lambda^k} \xrightarrow[k\to\infty]{} \frac{1-\lambda}{2+\lambda} > 0.
\]
Thus $\sum_{k=1}^\infty P(A_k)=\infty.$
 Clearly, the events $A_k,$ $k\ge1$ are independent. Therefore, by the second Borel-Cantelli lemma, $A_k$ occurs for infinitely many $k$ almost surely.

When $A_k$ occurs, let $m_k$ be the first time the walk hits the leaf $(k,k)$. Then $|X_{m_k}| = 2k$, and because the walk has not yet reached $(k+1,0)$, we have $N_{m_k} = k$. By Lemma~\ref{lem:bush_time}, $m_k -  S_k = O(k)$ almost surely. Therefore,
\[
\log m_k = \log ( S_k + O(k)) = \log  S_k + \log\left(1 + \frac{O(k)}{ S_k}\right),\quad\text{a.s..}
\]
From Lemma~\ref{lem:logSn}, we know that $\log  S_k \sim k \log(1/\lambda)$ almost surely as $k\to\infty$.
Thus, for $k$ large enough, we have $ S_k\ge \lambda^{-k/2}$ almost surely.
Hence, we obtain $$\log m_k \sim \log  S_k \sim k \log(1/\lambda), \text{ a.s.,}$$ as $k\to\infty.$ Consequently,
\[
\frac{|X_{m_k}|}{\log m_k} = \frac{2k}{\log m_k} \to \frac{2}{\log(1/\lambda)}, \quad \text{a.s.,}
\]as $k\to\infty.$
Hence $\limsup_{n\to\infty} |X_n|/\log n \ge 2/\log(1/\lambda)$ a.s.. Combined with the upper bound we get in \eqref{eq:ubnxn}, we obtain \begin{align*}
\limsup_{n\to\infty} \frac{|X_n|}{\log n} = \frac{2}{\log(1/\lambda)}, \quad \text{a.s..}
\end{align*}

For the next step, we show the lower limit.
By Proposition~\ref{thm:logscale},
\begin{align}
\frac{N_n}{\log n}\to\frac{1}{\log(1/\lambda)},\quad\text{a.s.,} \label{eq:Nn_as}
\end{align}
as $n\to\infty.$
Notice that for $k\ge0,$ $X_{ S_k}=(k,0)$, so $|X_{ S_k}|=k$ and $N_{ S_k}=k$.  Lemma~\ref{lem:logSn} gives $\log S_k\sim k \log(1/\lambda)$ almost surely, hence
\[
\lim_{k\to\infty}\frac{|X_{ S_k}|}{\log S_k}=\frac{1}{\log(1/\lambda)},\quad\text{a.s..}
\]
Therefore, we have
\begin{align}
\liminf_{n\to\infty}\frac{|X_n|}{\log n}\le\frac1{\log(1/\lambda)}, \quad\text{a.s..} \label{eq:upper}
\end{align}

We now prove the opposite inequality.  For each $k\ge1$, consider the walk starting from $(k,0)$.  Ignore the time spent in bushes and consider only the embedded backbone moves.  Each such move increases the backbone index by $1$ with probability $p=\frac{1}{1+\lambda}$ and decreases it by $1$ with probability $q=\frac{\lambda}{1+\lambda}$.  Define $D_k$ as the maximal decrease of the backbone coordinate before the walk first reaches $(k+1,0)$.  In other words, if $Z_0=k$, $Z_1,Z_2,\dots$ are the successive backbone positions (ignoring stays), and $T=\min\{m\ge1:Z_m=k+1\}$, then
\[
D_k = \max_{0\le t\le T}(k-Z_t).
\]
The event $\{D_k\ge d\}$ means that starting from $k$, the walk reaches $k-d$ before reaching $k+1$.  From \eqref{eq:gam}, we get
\begin{align*}
\mathbb{P}(D_k\ge d)=\frac{(q/p)^d-(q/p)^{d+1}}{1-(q/p)^{d+1}} = \lambda^d\,\frac{1-\lambda}{1-\lambda^{d+1}} \le \lambda^d. 
\end{align*}
Fix $c>1/\log(1/\lambda)$.  We have
\[
\sum_{k=1}^{\infty}\mathbb{P}\bigl(D_k > c\log k\bigr) \le \sum_{k=1}^{\infty} \lambda^{c\log k} = \sum_{k=1}^{\infty} k^{-c \log(1/\lambda)} < \infty.
\]
By the Borel--Cantelli lemma, almost surely $D_k \le c\log k$ for all sufficiently large $k$.  Consequently,
\[
\limsup_{k\to\infty}\frac{D_k}{\log k} \le c,\quad\text{a.s.}.
\]
Since $c>1/\log(1/\lambda)$ is arbitrary, we infer that
\begin{align*}
\limsup_{k\to\infty}\frac{D_k}{\log k} = \frac1{\log(1/\lambda)},\quad\text{a.s.}. 
\end{align*}
Thus taking \eqref{eq:Nn_as} into account, we  have
\begin{align}
\limsup_{n\to\infty}\frac{D_{N_n}}{\log n} = \limsup_{n\to\infty} \frac{D_{N_n}}{\log N_n}\cdot\frac{\log N_n}{\log n} \le \frac{1}{\log(1/\lambda)}\cdot 0=0, \quad \text{a.s..} \label{eq:Dn_log}
\end{align}
Since $S_{N_n}\le n<S_{N_n+1},$ we have
\begin{align}
|X_n| \ge N_n - D_{N_n}\quad\text{and}\quad N_n\ge X_{n,1}\ge N_n - D_{N_n},\quad \text{a.s.}. \label{eq:X_lower}
\end{align}
Dividing \eqref{eq:X_lower} by $\log n$ and using \eqref{eq:Nn_as} together with \eqref{eq:Dn_log} yields
\begin{align}
\liminf_{n\to\infty}\frac{|X_n|}{\log n} \ge \frac1{\log(1/\lambda)}\quad\text{and}\quad\lim_{n\to\infty}\frac{X_{n,1}}{\log n} =\frac1{\log(1/\lambda)},\quad\text{a.s.}. \label{eq:lower}
\end{align}
Taking \eqref{eq:upper} and \eqref{eq:lower} together, conclude that
\[
\liminf_{n\to\infty}\frac{|X_n|}{\log n} =\lim_{n\to\infty}\frac{X_{n,1}}{\log n} = \frac{1}{\log(1/\lambda)},\qquad\text{a.s.},
\]
which completes the proof of \eqref{eq:limb} and the lower limit in \eqref{eq:lslim}. \qed

\subsection{Proof of Theorem \ref{cor:limsup}: Case $\lambda>1$  }\label{sec:43}

\begin{proof}
Suppose $\lambda>1.$
Since $S_{N_n}\le n<S_{N_n+1},$ we have $X_{n,1}\le N_n.$ Thus, from Proposition~\ref{thm:logscale}, we get
\[
\limsup_{n \to \infty} \frac{X_{n,1}}{\log n} \leq \limsup_{n \to \infty} \frac{N_n}{\log n}= \frac{1}{\log \lambda}, \quad \text{a.s.}.
\]

For the reverse inequality, let $S_k$ be the first hitting time of the backbone vertex $(k,0)$. At time $S_k$, we have $X_{S_k,1} = k$. By Lemma~\ref{lem:logSn},
$
\frac{\log S_k}{k} \to \log \lambda$ almost surely as $k \to \infty,
$
so
\[
\frac{X_{S_k,1}}{\log S_k} = \frac{k}{\log S_k} \to \frac{1}{\log \lambda}, \quad \text{a.s.},
\]
as $k\to\infty.$
This implies
\[
\limsup_{n \to \infty} \frac{X_{n,1}}{\log n} \geq \frac{1}{\log \lambda}, \quad \text{a.s.}.
\]
Therefore, we conclude that
$$
\limsup_{n \to \infty} \frac{X_{n,1}}{\log n}= \frac{1}{\log \lambda},\quad\text{a.s..}$$\qed
\end{proof}

\section{Cutpoints and cut times: proof of Theorem \ref{thm:cutpoints}}\label{sec:E}

In this section, we consider the number of cutpoints and cut times, establishing the proof of Theorem \ref{thm:cutpoints}.

\proof
In this proof, we denote by $C$ the collection of all cutpoints of $(X_n)_{n\ge0}.$

From Definition \ref{def:cut}, we see that all bush vertices are not cutpoints and for $n\ge 1,$
 the backbone vertex $(n,0)$ is a cutpoint if and only if after its first visit, the walk never returns to $(n-1,0)$. Therefore, for $n\ge1,$
using the notation of \eqref{eq:pmainf}, we have
\begin{equation}\label{eq:cutprob}
P((n,0)\in C) = P_{n}(n-1,\infty,+) = 1 - \lambda.
\end{equation}

Define $\eta_i=\mathbf{1}_{\{(i,0)\in C\}}$ for $i\ge1$.  For $k\ge 2,$ we now consider indices $1\le i_1<i_2<\dots<i_k$. Using the strong Markov property, in view of \eqref{eq:pmab} and \eqref{eq:cutprob}, we have
\begin{align}\label{eq:etaj}
P(\eta_{i_1}&=1,\dots,\eta_{i_k}=1)=P((i_1,0)\in C,\dots,(i_k,0)\in C)\nonumber\\
&=P_{i_1}(i_1-1,i_2,+)\,P_{i_2}(i_2-1,i_3,+)\cdots P_{i_k}(i_k-1,\infty,+)\nonumber\\
&=(1-\lambda)^k\prod_{s=1}^{k-1}\frac{1}{1-\lambda^{i_{s+1}-i_s+1}},
\end{align}
which depends only on the gaps $i_{s+1}-i_s,$ $1\le s\le k.$ Hence the sequence $\{\eta_i\}_{i\ge1}$ is strictly stationary.

Next we prove that $\frac1n\sum_{i=1}^n\eta_i$ converges almost surely to $1-\lambda$ as $n\to\infty$. Since the sequence $\{\eta_i\}_{i\ge1}$ is strictly stationary, by the Birkhoff-Khinchin ergodic theorem (see e.g. Theorem 1 in \S3, Chap. V of Shiryaev \cite{s}),  the limit
\[
Y=\lim_{n\to\infty}\frac1n\sum_{i=1}^n\eta_i
\]
exists almost surely. To identify the limit, we show that the sequence is mean-ergodic, i.e., it converges in $L^2$ to the constant $E[\eta_1]$. To this end, notice that by \eqref{eq:cutprob} and \eqref{eq:etaj}, for $i
\ge 1, \ j\ge 1$ with $ i\ne j,$ we have
\begin{align*}
&\operatorname{Var}(\eta_i)=\lambda(1-\lambda),\\
&\operatorname{Cov}(\eta_i,\eta_j)=(1-\lambda)^2\left(\frac{1}{1-\lambda^{|j-i|+1}}-1\right)
=(1-\lambda)^2\frac{\lambda^{|j-i|+1}}{1-\lambda^{|j-i|+1}}\le \lambda^{|j-i|+1}.
\end{align*}
 Hence
\begin{align*}
&E\left[\left(\frac1n\sum_{i=1}^n\eta_i-E[\eta_1]\right)^2\right]
=\frac{1}{n^2}\sum_{i=1}^n\sum_{j=1}^n\operatorname{Cov}(\eta_i,\eta_j)\\
&\qquad=\frac{1}{n^2}\sum_{i=1}^n\operatorname{Var}(\eta_i)+\frac{2}{n^2}\sum_{i=1}^{n}\sum_{j=i+1}^n\operatorname{Cov}(\eta_i,\eta_j)\\
&\qquad\le \frac{n\lambda(1-\lambda)}{n^2}+\frac{2(1-\lambda)}{n^2}\sum_{i=1}^{n}\sum_{j=i+1}^n\lambda^{j-i+1}\\
&\qquad\to 0,
\end{align*}
as $n\to\infty.$
Thus $\frac1n\sum_{i=1}^n\eta_i$ converges to $E[\eta_1]=1-\lambda$ in $L^2$ as $n\to\infty$. Since the $L^2$ limit is a constant, we conclude $Y=1-\lambda$ almost surely. Therefore,
\begin{align}\label{eq:limcn}
\lim_{n\to\infty}\frac{C(n)}{n}=\lim_{n\to\infty}\frac1n\sum_{i=1}^n\eta_i=1-\lambda, \quad\text{a.s..}
\end{align}
The first part of Theorem~\ref{thm:cutpoints} is proved.

Now we turn to the second part of Theorem~\ref{thm:cutpoints}. Recall that $M(N)$ is the number of cut times up to time $N$. Let $n_N = \max\{n\ge 1 : S_n \le N\}$ be the number of distinct backbone vertices visited by time $N$ (with $n_N=0$ if $S_1>N$). Then it is clear that
\[
M(N) = \sum_{i=1}^{n_N} \mathbf{1}_{\{(i,0)\in C\}}.
\]
 Notice that by  Proposition~\ref{thm:logscale},
\[
\lim_{N\to\infty}\frac{n_N}{\log N}=\frac{1}{\log(1/\lambda)},\quad\text{a.s.}.
\]
Moreover, it follows from \eqref{eq:limcn} that
\[
\frac{1}{n_N}\sum_{i=1}^{n_N}\mathbf{1}_{\{(i,0)\in C\}}\to 1-\lambda,\quad\text{a.s.,}
\] as $N\to\infty.$
Therefore,
\[
\frac{M(N)}{\log N}=\frac{1}{n_N}\sum_{i=1}^{n_N}\mathbf{1}_{\{(i,0)\in C\}}\cdot\frac{n_N}{\log N}
\to \frac{1-\lambda}{\log(1/\lambda)},\quad\text{a.s.},
\]
as $N\to\infty.$ This completes the proof of Theorem~\ref{thm:cutpoints}. \qed
\vspace{.5cm}

\noindent{{\bf \Large Acknowledgements:}} The authors would like to thank Professor W. M. Hong for some useful suggestions and comments.

\end{document}